\documentclass[12pt]{amsart}
\usepackage{amsmath,amsthm,amsfonts,amscd,amssymb,eucal,latexsym}

\newcommand{\ra}{\rangle}
\newcommand{\la}{\langle}

\numberwithin{equation}{section}

\theoremstyle{plain}
\newtheorem{lem}{Lemma}[section]
\newtheorem{thm}[lem]{Theorem}
\newtheorem{cor}[lem]{Corollary}
\newtheorem{pro}[lem]{Proposition}
\newtheorem{exa}[lem]{Example}
\newtheorem{defn}[lem]{Definition}
\theoremstyle{definition}

\theoremstyle{remark}

\def\ca{{\mathcal A}}

\def\cd{{\mathcal D}}
\def\ce{{\mathcal E}}
\def\cf{{\mathcal F}}

\def\cam{{\mathcal M}}

\def\bn{{\mathbb N}}

\def\br{{\mathbb R}}

\def\bt{{\mathbb T}}
\def\bz{{\mathbb Z}}

\def\a{\alpha}
\def\eps{\varepsilon}

\def\d{\delta}

\numberwithin{equation}{section}

\def\ca{{\mathcal A}}

\def\cd{{\mathcal D}}
\def\ce{{\mathcal E}}
\def\cf{{\mathcal F}}

\def\cam{{\mathcal M}}

\def\bn{{\mathbb N}}

\def\br{{\mathbb R}}

\def\bt{{\mathbb T}}
\def\bz{{\mathbb Z}}

\def\a{\alpha}
\def\eps{\varepsilon}

\def\d{\delta}

\def\f{\varphi}

\begin{document} 

\title{ON WEAKLY $D-$DIFFERENTIABLE OPERATORS}
\author[E.~Christensen]{Erik Christensen}
\address{\hskip-\parindent
Erik Christensen, Institut for Matematiske Fag, University of Co-penhagen, Demark.}
\email{echris@math.ku.dk}

\date{\today}

\begin{abstract}
 Let $D$ be a self-adjoint operator on a Hilbert space $H$ and $a$ a bounded operator on $H.$ We say that $a$ is weakly $D-$differentiable, if for any pair of vectors $\xi, \eta$ from  $H$ the function  $ \la e^{itD}ae^{-itD}\xi, \eta\ra$ is differentiable.  We give an elementary example of a bounded  operator $a,$ such that $a$ is weakly $D-$differentiable, but the function  $  e^{itD}ae^{-itD}$ is not uniformly differentiable.
 We show that  {\em weak $D-$differentiability } may be characterized by several other properties, some of which are related to the commutator $(Da - aD).$
\end{abstract}

\maketitle

\section{Introduction}

In mathematical physics and especially in quantum mechanics
commutators between operators on Hilbert space have played a
fundamental role for about 100 years.  People are mostly concerned with an expression like $i[D,a],$
where $D$ is an unbounded self-adjoint operator,  and $a$ is an operator
representing an observable.

Since the appearance of the papers \cite{K} and \cite{S}
by Kadison and Sakai respectively, the research on bounded and
unbounded derivations on C*-algebras, \cite{KR},  took off, and an impressive amount
of research has been published. As far as we know, the article
\cite{PS}, was the first which dealt with unbounded derivations on
C*-algebras.

  It is desirable to be able to perform commutators  like $[D,a]$
 inside the mathematical discipline { \em
    noncommutative geometry} \cite{AC}. The reason is, that in Connes'
  set up of noncommutative geometry the space is replaced by an
  algebra of operators on a Hilbert space, and derivatives of functions are replaced by commutators of the type $[D,a].$

Usually the commutator $i[D,a]$ appears as the derivative with respect to the norm topology at $t=0$ of the function $  e^{itD}ae^{-itD},$ and here  we meet a part of the problem which inspired this investigation. Actually the norm  derivative, or uniform derivative as we prefer to say, is a bounded and everywhere defined operator, whereas the commutator $i(Da-aD)$ is only defined on dom$(D)\cap \{\xi \in H\, \big| \, a\xi \in \mathrm{dom}(D) \, \},$ which may not even be a dense subset of $H.$ 

 In order to overcome some of the problems with respect to the domain of definition for sums of unbounded operators, Barry Simon suggested in \cite{BS} that one should look at forms rather than at operators. For any bounded operator $a$ on $H$ it is possible to define a sesquilinear form on dom$(D)$ by 
$$\forall \xi, \eta \in \mathrm{dom}(D):\quad S(i[D,a])(\xi, \eta) := i(\la a\xi, D\eta\ra - \la aD\xi, \eta\ra).$$
Sometimes this form is bounded and is implemented by a bounded operator $b$ such that
$$\forall \xi, \eta \in \mathrm{dom}(D):\quad S(i[D,a])(\xi, \eta) = \la b\xi, \eta\ra.$$ We have not been able to find a systematic presentation of the results which relate to this phenomenon. We know that many people know many details and aspects of this set up, but we have given up on the project to find out who presented this or that property first.
It is our hope that this article may help to clarify and distinguish between some properties, which are related to the function $e^{itD}ae^{-itD}.$ We concentrate on various kinds of differentiability at the point $t = 0,$ which we list below.\begin{defn} \label{DifProp}
\begin{itemize}
\item []
\item [] We say that a bounded operator $a$ is uniformly  $D-$differentiable if  there exists a bounded operator $b$ on $H$ such that $$ \qquad \qquad \underset{t \to 0}{\lim} \|\frac{e^{itD}ae^{-itD} -a}{t} - b\| = 0. $$ The operator $b$ is called the uniform  $D-$derivative of $a$ and denoted $\d_u^D(a).$  
\item []  We say that a bounded operator $a$ is strongly  $D-$differentiable if there exists a bounded operator $b$ on $H$ such that $$\forall \xi \in H:\quad \underset{t \to 0}{\lim} \|\big(\frac{e^{itD}ae^{-itD} -a}{t} - b\big)\xi\| = 0. $$
\item []  We say that a bounded operator $a$ is weakly  $D-$differentiable if there exists a bounded operator $b$ on $H$ such that $$\forall \xi, \eta  \in H:\quad \underset{t \to 0}{\lim} |\langle \big(\frac{e^{itD}ae^{-itD} -a}{t} - b\big)\xi, \eta \rangle| = 0. $$
The operator $b$ is called the weak $D-$derivative of $a$ and denoted $\d_w^D(a).$   
\end{itemize}
\end{defn}

The properties above are listed according to strength, with the stron-gest at the top. 
The present article is based on the observation that there exists an elementary example, which shows that a weakly $D-$dif-ferentiable operator may not be uniformly $D-$differentiable.  One of the referees then asked about the concept {\em strong $D-$differentiability}, and it turns out, as we shall see, that this property is equivalent to {\em weak $D-$differentiability.}  

We have found several other properties which are equivalent to weak $D-$ differentiability and we list them here without all the exact definitions.
\begin{itemize}
\item[(i)] The function $ e^{itD}ae^{-itD}$ is Lipschitz continuous with respect to the operator norm.
\item[(ii)]The sesquilinear form $S(i[D,a])(\xi, \eta)$ on dom$(D)$ is bounded.
\item[(iii)] The infinite matrix commutator $m(i[D,a])$ is the matrix of a bounded operator. 
\item[(iv)] The commutator $i[D,a] $ is defined on dom$(D)$ and bounded.
\end{itemize}

The domain of definition  for $\d_w^D $ is a subalgebra of $B(H),$ which we denote by  $\cd_w^D,$ and $\d_w^D $ becomes a derivation of this algebra into $B(H).$  The domain of definition for $\d_u^D $ is denoted $\cd_u^D. $ It is a subalgebra of 
$\cd_w^D $ and $\d_u^D $  is a derivation of this algebra into $B(H).$ 

This article was preceded by a preprint which contains most of the results presented here, but the emphasis there  was laid on the space of infinite matrices which are naturally associated with $D.$ We still think that this point of view is fruitful since it provides a setting inside which higher commutators like $[D, [D, \dots ,[D, a]\dots ]]$ make sense as matrices even if none of the matrices represent bounded operators. Later we realized that the concept named weak $D-$differentiability probably is the one among all the equivalent ones, which is the most natural to focus on.  

We end the article by a characterization of those weakly $D-$differen-tiable operators, that are uniformly $D-$differentiable and it turns out that for a natural number $k$ a bounded operator $a$ is $k$ times uniformly $D-$differentiable if it is $k+1$ times weakly $D-$differentiable.

We are happy to be able to thank Alain Connes, Marc A. Rieffel, Barry Simon and the referees for valuable comments and most relevant suggestions for improvements.

\section{Uniform $D-$differentiability}
The theory of one parameter groups of isometries on Banach spaces is very important and applied in many different settings. Here we will only rely on the most elementary parts of this theory. 

Given a self-adjoint operator $D$ on a Hilbert space $H$ we can define a strongly continuous one parameter   unitary group $u_t$ on $H$ by $u_t := e^{itD} $ and a one parameter group $\a_t$  of *-automorphisms on $B(H)$ by $\a_t(a) := u_tau_{-t}.$ The group $\a_t$ is not in general strongly continuous on $B(H), $ and that aspect is  actually the basic reason, why this article may have a life of its own, as the result in Theorem \ref{CSD} shows.  

We start by extending the differentiability from one point to all of $\br.$

\begin{lem} \label{trans}
Let $a$ be a bounded uniformly $D-$differentiable  operator then for any real $s$ the operator $\a_s(a)$ is uniformly $D-$differentiable and the function $s \to \a_s(a)$ is uniformly differentiable everywhere such that 
$$ \d_u^D (a_s(a)) \, = \,\frac{d}{ds}\a_s(a)\, = \,  \a_s(\d_u^D(a)).$$
 
\end{lem}
 \begin{proof}
For any real $s$ the automorphism $\a_s$ is an isometry on $B(H)$, and for any pair of reals $s, t$ we have $\a_t \circ \a_s = \a_{t+s},$  hence for any reals $s$ and $t $ we have 
\begin{align*} \underset{t \to 0}{\lim} &\| \frac{\a_t(\a_s(a))- \a_s(a)}{t} - \a_s(\d_u^D(a)\| \\= \underset{t \to 0}{\lim}&\|\frac{\a_{t+s}(a)- \a_s(a)}{t} - \a_s(\d_u^D(a)\| \\=  \underset{t \to 0}{\lim} &\|\a_s\big(\frac{(\a_t(a))- a)}{t} - \d_u^D(a)\big)\|  \, = \, 0.\end{align*}
 \end{proof} 

We collect some of the well known properties for $\d_u^D.$ 
\begin{thm} \label{du}
\begin{itemize}
\item[]
\item[(i)] For a bounded uniformly $D-$differentiable operator $a$ and a vector $\xi$ in $\mathrm{dom}(D)$
$$ a\xi \in \mathrm{dom}(D) \text{ and } i(Da-aD)\xi = \d_u^D(a)\xi.$$
\item[(ii)] The operator $\d_u^D$ is closed and its domain of definition is dense in $B(H)$ in the strong operator topology. 
The domain $\cd_u^D $ is in general not norm dense in $B(H).$ 
\item[(iii)] The linear space $\cd_u^D := \mathrm{dom}(\d_u^D)$ is a Banach *-algebra under the norm $\|a\|_{D}:= \|a\| + \|\d_u^D(a)\|,$ and  $\d_u^D$ is a derivation into $B(H).$ 
\end{itemize} \end{thm}
\begin{proof}

Item (i)  follows from the fact that the operator $iD$ is the generator of the one parameter unitary group $u_t ,$ but more details will follow in the next section, when we discuss the similar  result for a weakly $D-$differentiable operator. 

Item (ii) has a standard  proof based on the fundamental theorem of calculus. The arguments we will use will be reused in the proofs of the Theorems \ref{WD} and \ref{CSD}, so we present the details here.  Suppose $(a_n)$ is a  convergent sequence of uniformly $D-$differentiable operators such that $$ a_n \to a \text{ and } b_n :=  \d_u^D(a_n ) \to b \text{ for } n \to \infty,$$
Then for every $n $ there exists a function $t \to \eps_n(t) \in B(H), $ which is continuous at $0$ and satisfies $\eps_n(0) = 0,$ such that 
$$ \forall t  \in \br : \quad \a_t(a_n) = a_n + tb_n + t\eps_n(t).$$ By replacing $t$ by $-t$ and  applying $\a_t$ we get 
$$ \forall t  \in \br : \quad a_n = \a_t(a_n)  - t\a_t(b_n) - t\a_t(\eps_n(-t)). $$ By adding the 2 equations, rearrangements and division by $t$ we get
\begin{equation} \label{bCont} \forall t \in \br : \quad \a_t(b_n)- b_n  = \eps_n(t) - \a_t(\eps_n(-t)).
\end{equation} 

This shows that the function $\a_t(b_n)$ is continuous with respect to the norm topology on $B(H)$ at $t = 0. $ Let $s, t$ be reals then $\a_{s+t}(b_n) = \a_t(\a_s(b_n))$ and since $\a_s$ is an isometry we get 
$$ \|\a_{s+t}(b_n) - \a_s(b_n) \| = \|\a_s\big(\a_{t}(b_n) - b_n\big) \| = \|\a_{t}(b_n) - b_n \|,$$ and we see that the function $\a_s(b_n) $ is uniformly continuous on $\br.$ 
Since $\|\a_t(b) - \a_t(b_n)\| = \|b - b_n\|$ the function $t \to \a_t(b)$ is the uniform limit of a sequence of continuous functions and hence  also a  continuous function. Then the Riemann integral $\int_0^t \a_s(b)ds$ exists and  we may define an operator valued  norm differentiable function  $c(t)$ by $$c(t) : = a + \int_0^t\a_s(b)ds.$$
On the other hand we know that the function $\a_t(a_n) $ is norm differentiable with derivative $\a_t(b_n),$ so for each $n$  
$$ \a_t(a_n) = a_n + \int_0^t\a_s(b_n)ds.$$
Hence  $c(t) = \a_t(a)$ and $a$ is uniformly $D-$differentiable with $\d_u^D(a) = b.$

To show that the domain of definition for $\d_u^D$ is strongly dense in $B(H)$ we let $E_n$ denote the spectral projection for $D$ corresponding to the interval $[-n,n], $ then $DE_n$ is a bounded operator defined on $H$ and  for any bounded operator $c$ on $H$ we have
$$e^{itD}E_ncE_ne^{-itD} = e^{it(DE_n)}E_ncE_ne^{-it(DE_n)}$$ is uniformly differentiable at $t=0$ with derivative $i[DE_n, E_ncE_n].$  Hence $E_ncE_n$ is in the domain of $\d_u^D$ with $\d_u^D(E_ncE_n) = i[DE_n, E_ncE_n] . $  Since the sequence $(E_ncE_n)$ converges to $c$ in the strong operator topo-logy, it follows that the domain of $\d_u^D$ is strongly dense in $B(H).$  In the Example \ref{No} we show that the domain of  $\d_u^D$ may not be norm dense in $B(H).$

Item (iii): Let $a, b $ be uniformly $D-$differentiable operators, then the mapping $t \to \a_t(b)$ is continuous at $t=0$ by assumption, so the standard  identity below,  valid for $t \neq 0$  
$$\frac{\a_t(ab) - ab}{t}  = \frac{\a_t(a) - a}{t}\a_t(b) + a\frac{\a_t(b) - b}{t }$$ shows that $$ ab \in \mathrm{dom}(\d_u^D) \text{ and } \d_u^D(ab) = \d_u^D(a)b + a \d_u^D(b).$$
Then the space of uniformly $D-$differentiable operators is an algebra, and $\d_u^D$ is a derivation on this algebra. The Banach space property follows from the fact that $\d_u^D$ is a closed derivation, and the submultiplicativity of the norm $\|a\|_D$ from the derivation property. The *-operation is just inherited from $B(H).$  
\end{proof}

\section{Weak and other notions of $D-$differentiability}

We show by an example that weak and uniform  $D-$differentiability are different properties. Then we define some other types of $D-$differen-tiability which we have met in the literature and one we think is original. Finally we show that they are all equivalent to weak $D-$differentia-bility and that strong and weak $D-$differentiability are equivalent too.
We also show that the operator $\d_w^D$ is closed, its domain of definition is a Banach *-algebra under a suitable norm, and it is a derivation on this algebra. Most of this may be known to many people, but we have not been able to find an explicit presentation of the equivalences of all the properties. 

\subsection{ Weak but not uniform  $D-$differentiability.} \label{WneqU}
We show by an example based on classical Fourier analysis that {\em weak $D-$differentiabi-lity  and  uniform $D-$differentiability are non-equivalent properties.  }

We let $\bt$ denote the unit circle with normalized Lebesgue measure 
$(1/(2\pi))m,$ $H$ the Hilbert space $L^2(\bt, m/(2\pi))$ and  $D = \frac{1}{i} \frac{d}{d\theta}$ 
the differentiation with respect to arclength times $-i.$  To be more precise we let the symbol $ D:= \frac{1}{i}\frac{d}{d\theta}$  denote the self-adjoint operator on $L^2(\bt, m/(2\pi))$  with pure point spectrum $\bz$ such that the functions $\f_n(z) = z^n, $ $n$ in $\bz$ form  an orthonormal basis of eigenvectors for $D$ with $D\f_n = n\f_n.$   The unitaries $u_t:= e^{itD} $ act on a function $f(z)$ in $L^2(\bt)$ such that $(u_tf)(z)) = f(e^{it}z),$ and we will let $f_t(z)$ be defined by $f_t(z) : = u_t f(z) = f(e^{it}z).$     

We will look at  multiplication operators $M_f$ where $f(z)$
is an essentially bounded measurable function on $\bt$ and it
is well known that $\|M_f\| = \|f\|_{\infty}.$ Then for  the automorphisms $\a_t $ of $B(H)$ which are implemented by $u_t$ we get for an essentially bounded measurable function $f(z)$ and a square integrable  function $g(z)$ that  $$\a_t(M_f)g = u_t (fg_{-t}) = f_tg = M_{f_t}g ,$$ so $\a_t(M_f) = M_{f_t}.$
We define the function $\theta(z) $ on $\bt$ 
by $\theta(z) =  \mathrm{Arg}(z),$ which is the   argument of the non-zero complex number $z$ in the interval $]-\pi, \pi].$ The operator $ a$ which will be shown to be weakly, but not uniformly $D-$differentiable, is the multiplication operator   induced by the function $|\theta|(z),$ which for $z$ in $\bt$ is given by $|\theta|(z) := |\theta(z)| = |\mathrm{Arg}(z)|.$ 
 The function $|\theta|(z)$ is continuous  on $\bt$ and  differentiable with respect to arclength except at the points $\{-1, 1\}.$ The function sign$(\theta)$ which we define below is equal to the derivative of $|\theta|$ except at the points $\{-1, 1\}.$
 \begin{displaymath} \text{sign}(\theta)(z) := \begin{cases} -1 \text{ for } - \pi < \mathrm{Arg}(z) < 0 \\ \,\, \,\,  1 \text{ for } \, \, \, \, \, \, \, 0 < \mathrm{Arg}(z) < \pi \\
\, \,\, \, \,  0  \text{ for } \,\,\,\, \, \, \,  z \in \{-1, 1 \} \end{cases}.\end{displaymath}

This lack of continuity for sign$(\theta)$ shows that there is no limit in the uniform norm for the set of  quotients $(|\theta|_t - |\theta|)/t$ for $t \to 0.$ Since the mapping $f \to M_f$ is an isometry and $$\big(\a_t(M_{|\theta|}) - M_{|\theta|}\big)/t =M_{\big((|\theta|_t -|\theta|)/t\big)},$$ 
we find that there is no limit in the operator norm for the quotients $\big(\a_t(M_{|\theta|}) - M_{|\theta|}\big)/t$ for $t \to 0.$  Hence $M_{|\theta|} $ is not uniformly $D-$differenti-able. 
 
 In order to show that $|\theta| $ is weakly $D-$differentiable we first remark, that the quotients $(|\theta|_t - |\theta|)/t $ defined for $0 < |t| < \pi$ are all bounded continuous functions of  norm $1,$ and for $0 < t < \eps < \frac{\pi}{2} $  we have $$ \forall z \in  \{ e^{i\theta}\, \big| \,  \theta \in [-\pi + \eps, - \eps] \cup   
 [\eps, \pi - \eps]\}: \, \, \,  (|\theta|_t - |\theta|)(z)/t  = \mathrm{sign}(\theta)(z).$$    Hence for any pair $ f(z), g(z)$ of square integrable functions on $\bt$  the product $ f(z) \overline{g(z)} \big((|\theta|_t - |\theta|)/t\big)(z)$ is an integrable function on  $\bt,$ which is numerically bounded by the positive integrable function $|f(z)||g(z)|.$ Hence by Lebesgue's theorem  on dominated convergence there exists a limit for $t \to 0,$  such that 
 \begin{align*}
&\underset{t \to 0} {\lim} \frac{1}{t}\la \big(M_{|\theta|_t} - M_{|\theta|}\big)f, g \ra  \\= & \underset{t \to 0} {\lim} \frac{1}{t}\frac{1}{2\pi} \int_{\bt} f(z) \overline{g(z)} \big(|\theta|(e^{it}z) - |\theta|(z)\big) \, dm(z) \\  = & \frac{1}{2\pi}\int_{\bt} f(z) \overline{g(z)} \mathrm{sign}(\theta)(z)\,dm(z)  =\la M_{\mathrm{sign}(\theta)}f, g \ra \end{align*} 
 
 Hence $M_{|\theta|}$  is weakly differentiable with $\d_w^D(M_{|\theta|}) = M_{\mathrm{sign}(\theta)}.$

\subsection{$D-$Lipschitz continuity}
\begin{defn}
A bounded operator $a$ on $H$ is said to be $D-$Lipschitz continuous  
if $$ \underset{t \to 0}{\lim \inf}\|\frac{\a_t(a) -a}{t}\| < \infty.$$
\end{defn} 

After we have presented the main theorem of this section, it becomes clear that a bounded  operator $a$ is $D-$Lipschitz continuous if and only if there exists a positive constant $c$ such that $$\forall t \in \br : \quad \|\a_t(a) - a\| \leq c |t|.$$

This means that in the setting where we look at multiplication operators on $L^2(\bt)$ and the self-adjoint operator $D = \frac{1}{i}\frac{d}{d\theta},$ a multiplication operator $M_f$ is $D-$Lipschitz continuous if and only if   $M_f$ may be given by a continuous function $f,$  and there exists a positive constant $c$ such that $$ \forall z \in \bt\, \forall t \in  \br : \quad |f(ze^{it}) - f(z)| \leq c|t|.$$

\subsection{Bounded $D-$form} \label{BdForm} 

Let $a$ be a bounded operator on $H$ and $\xi, \eta$  vectors in the domain of definition for $D,$ then the strong continuity of the one parameter unitary group $u_t = e^{itD}$ and  the manipulations below show that the limit $\underset{t \to 0}{\lim} \la(\a_t(a) - a) \xi, \eta \ra/t$ exists. Below and in the rest of this article the symbol $I$ denotes the identity operator on $H.$ 
\begin{align*}\frac{\la (\a_t(a) - a)\xi, \eta \ra}{t} = & \frac{\la a\xi, (u_{-t} - I)\eta \ra}{t} +\frac{\la a( u_{-t} - I)\xi, u_{-t}\eta \ra}{t} \\ \text{so }\underset{t \to 0}{\lim} \la(\a_t(a) - a) \xi, \eta \ra/t  =  & i \la a \xi, D \eta \ra - i \la aD\xi, \eta \ra.\end{align*}  Based on this we may make the following definitions.
 
\begin{defn} Let $a$ be in $B(H)$ the  sesquilinear form  $S(i[D,a])$  on $\mathrm{dom}(D)$ is  given as 
\begin{align*}\forall \xi, \eta \in \mathrm{dom}(D)\quad S(i[D,a])(\xi, \eta) := &i \la a \xi, D \eta \ra - i \la aD\xi, \eta \ra \\ = &\underset{t \to 0}{\lim} \la(\a_t(a) - a) \xi, \eta \ra/t. \end{align*}
\end{defn} 

\begin{defn} \label{bdform}
A bounded operator $a$ on $H$ is said to have a bounded $D-$form if  the sesquilinear form $S(i[D,a])$ on $\mathrm{dom}(D)$ is bounded. 
\end{defn}

\subsection{Bounded $D-$matrix commutator} 
 Given a bounded or unbounded self-adjoint operator $D$ on a  Hilbert space $H,$ then we will define a sequence of pairwise orthogonal projections $(e_n)_{n \in \bz}$ with strong operator sum $I$ by letting $e_n$ be the spectral projection for $D$ corresponding to the interval $]n-1, n].$ We will let $\cam $ denote the linear space consisting of all infinite matrices $ y = (y_{rc})_{r,c \in \bz}$ where $ y_{rc}$ is an operator in
 $e_rB(H)e_c.$ In this way it is possible to define a linear mapping
 $m: B(H) \to \cam $ by $m(a)_{rc} := e_r a e_c.$ We will also equip
 $\cam$ with the matrix product, whenever it gives a
 bounded operator at all entries, but we do not demand that the product must be the matrix of a bounded operator.    Also $D$ has an interpretation as an
 element $m(D)$ in $\cam.$ In order to give this definition we define
 for each $n$ in $\bz$ a bounded self-adjoint operator $d_n$ in $e_nB(H)e_n$ by $d_n := De_n,$ and then  we can define the canonical representative $m(D) $
 for $D$ in $\cam $ by
 \begin{displaymath} m(D)_{rc} = \begin{cases} 0 \text { if } r \neq c 
\\ d_c \text{ if } r = c\end{cases} . \end{displaymath} 
 It is quite obvious that for any matrix $y = (y_{rc})$ in $\cam,$ 
the commutator $[m(D), y] $ has a meaning, and it is the matrix given by 
\begin{displaymath} [m(D), y ]_{rc} := d_ry_{rc}  - y_{rc}d_c. \end{displaymath}  
Hence for any operator $a$ in $B(H)$ we can define the $D-$matrix commutator $m([D,a])$
as the infinite matrix in $\cam$  given by
$$   m([D, a])_{rc} \,   =  \,d_rae_c - e_r a d_c.$$

It should be noticed, that here we are discussing the commutator $m([D,a]),$ while we will consider 
$im([D,a]),$ when we will apply this matrix in the study of weak $D-$differentiability.

The main reason why we introduce these infinite matrices is, that the expression $m([D,a]),$  makes sense for any bounded operator $a,$ but also the higher commutators like $m([D,[D, \dots ,[D,a]\dots ]])$ all have an obvious interpretation as elements in $\cam.$  
An element $y =(y_{rc})$ in $\cam$ is said to be bounded, if there exists a bounded operator $b$ on $H$ such that for each pair $rc$ we have $y_{rc} = e_rbe_c.$
The space $\cam$ will be equipped with a *-operation which is an
extension of the classical one defined on bounded matrices, so we
define for an element $ y = (y_{rc}) $ of $\cam $ the element $y^* $ 
as the matrix $(y^*)_{rc} = (y_{cr})^*. $ Then it is clear that $y^*$
is  bounded  if and only if $y$ is bounded. We can now define the concept we call bounded $D-$matrix commutator.

\begin{defn} 
Let $a$ be a bounded operator on $H,$ then we say that $a$  has a bounded $D-$matrix commutator if $m([D,a])$ is bounded.
 \end{defn}
We end this subsection by an example, which is based on the example given in Subsection \ref{WneqU}. We will use the same notation without no further comments. The operator $D$ has the set $\{\f_n : n \in \bz\}$ of eigenvectors,  such that $D\f_n = n\f_n.$ Then the spectral projection $e_n$ is one dimensional, and  given as $e_n\xi = \la\xi, \f_n\ra \f_n.$  
\begin{exa} \label{NotMatrixBd}
There exists  a bounded function $f(z)$ on $\bt$   which is differentiable with respect to arclength everywhere except at 1 point, but the corresponding multiplication operator does not have a bounded $D-$matrix commutator.
\end{exa}

Recall that $\theta(z) := \mathrm{Arg}(z), $  then the  Fourier series for $\theta(z)$ equals $\sum_{n \neq 0} i((-1)^{n} /n)\f_n.$ The corresponding multiplication operator $M_\theta$ is bounded and we can compute the corresponding matrix from the equation 
\begin{displaymath}
\langle M_{\theta} \f_c, \f_r\rangle \, = \, \frac{1}{2\pi}\int_{\bt}\theta(z)z^{c-r}dm(z)\, = \, \begin{cases} 0 \qquad  \, \, \,  \text{ if } r=c \\
\frac{i(-1)^{r-c}}{r-c} \text{ if } r \neq c \end{cases},\end{displaymath}
and the elements in the matrix $m(M_{\theta})$ are given by  \begin{displaymath} m(M_\theta)_{rc} = \begin{cases} 0 \qquad \qquad \qquad \qquad  \text{ if } r =c \\ i\frac{(-1)^{r-c}}{r-c}e_rM_{(z^{(r-c)})}e_c \text{ if } r \neq c \end{cases}. \end{displaymath}  
The matrix  $im([D,a]),$ is given by 
 \begin{displaymath} i\,m([\frac{1}{i}\frac{d}{d\theta}, M_\theta])_{rc} = \begin{cases} 0 \qquad \qquad \qquad \qquad \quad \, \,  \text{ if } r =c \\ (-1)^{(r-c+1)} e_rM_{(z^{(r-c)})}e_c\text{ if } r \neq c \end{cases}. \end{displaymath}  
This is {\em not the matrix of a bounded operator  } $b$ on $H.$ Suppose it were then the elements   $(b_{r0})_{r \in \bz}$ in the column number $0$ of $b$'s matrix  would satisfy 
$$\forall N \in \bn: \, \sum_{k=-N}^{N} b_{k0}^*b_{k0} \leq \|b\|^2 e_0, \text{ but }  \sum_{k=-N}^{N} b_{k0}^*b_{k0} =  2N e_0.$$

  The main reason why the operator  $M_{\theta}$ is not weakly $\frac{1}{i}\frac{d}{d\theta}-$differenti-able is that the function $\theta $ does not belong to the domain of definition for the self-adjoint operator $\frac{1}{i}\frac{d}{d\theta}.$

\subsection{Bounded $D-$commutator} 
The  last property, we will study,  is well known and  fulfilled for uniformly $D-$differentiable operators. Before we present the definition below, we  recall the notion of a core for a closed unbounded   operator $T.$ A linear subspace $\ce$ of $H$ is said to be a core for $T$ if $\ce$ is contained in the domain of definition for $T$ and the closure of the restriction of $T$ to $\ce$ equals $T.$ 

\begin{defn} \label{BdCo} 
Let $a$ be in $B(H),$ then we say that $a$ has a bounded $D-$commutator if the operator $iDa -iaD$ is defined and bounded on a core for $D.$
\end{defn}
Below we extend  the example \ref{NotMatrixBd}, and we may  conclude that the {\em core } property in the Definition \ref{BdCo} is essential. 

\begin{exa} The operator $M_{\theta} $ has the property that the commutator  $i[\frac{1}{i}\frac{d}{d\theta},M_{\theta}]$ is a bounded operator on a dense subspace of $\mathrm{dom}(D),$ but the sesquilinear form $S(i[\frac{1}{i}\frac{d}{d \theta},M_{\theta}])$ is unbounded on $\mathrm{dom}(D).$  \end{exa}

Let $\ce$ denote the linear span of the eigenfunctions $\{z^n\, | \, n \in \bz \}$ of $\frac{1}{i} \frac{d}{d\theta}.$ The element in $\cam$ given as $m(i[\frac{1}{i} \frac{d}{d\theta},M_{\theta} ]) $ is unbounded, so it follows that the sesquilinear form $S(i[\frac{1}{i}\frac{d}{d\theta},M_{\theta}])$ is unbounded on $\ce$ and hence on dom$(\frac{1}{i}\frac{d}{d\theta}).$  Let $\cf:= \{ g\in \ce \, \big|\,  g(-1) =0\},$ then
$\cf$ is dense in $L^2(\bt)$ and for $g$ in $\cf$  we have that $M_{\theta}g$ is in dom$(D)$ and $$\frac{d}{d\theta}(M_{\theta}g) -M_{\theta}\frac{d}{d\theta}g = g, $$
so the commutator equals the identity operator on the subspace $\cf.$

We will now show that all the properties discussed above are equivalent.

\begin{thm} \label{WD}
Let $a$ be a bounded operator on $H.$ The following properties are equivalent:
\begin{itemize}
\item[(i)] $a$ is strongly $D-$differentiable.
\item[(ii)] $a$ is weakly $D-$differentiable.
\item[(iii)] $a$ is $D-$Lipschitz continuous.
\item[(iv)] The sesquilinear form $S(i[D,a])$ on the domain of $D$ is bounded.
\item[(v)]  The infinite matrix $m(i[D, a]) $ represents a bounded operator.
\item[(vi)] The operator $Da - aD$ is bounded and its domain of definition is a core for $D.$
\item[(vii)] The operator $Da - aD$ is bounded and its domain of definition is  dom$(D).$ 
\end{itemize}
If $a$ is weakly $D-$differentiable then  
\begin{align*} \forall \xi, \eta \in H \quad  \underset{t \to 0}{\lim} \frac{\la (\a_t(a) - a)\xi, \eta\ra } {t} = & \la \d_w^D(a) \xi, \eta\ra  \\ a\mathrm{dom}D \subseteq \mathrm{dom}D  \text{ and }  \d_w^D(a)\big| \mathrm{dom}(D) = & i(Da - aD) \\
\forall t \in \br: \quad \|\a_t(a) - a\| \leq & \|\d_w^D(a)\| |t|.
\end{align*}
\end{thm}

\begin{proof}
The proof is organized such that we prove (i) $\Rightarrow $ (ii)  $\Rightarrow $ (iii) $\Rightarrow $ (iv) $\Rightarrow $ (v)
$\Rightarrow $ (vi) $\Rightarrow $ (vii)
  $\Rightarrow $ (i).  

Proof of (i) $\Rightarrow$ (ii):

This is a consequence of the Cauchy-Schwarz  inequality. 
 
Proof of (ii) $\Rightarrow $ (iii):

This is a consequence of the uniform boundedness principle. We will apply this principle to the set $F := \{(\a_t(a) - a)/t\, \big|\, t \neq 0 \}.$  Let $\xi, \eta$ be vectors in $H$ then since the function $ g(t):= \la \a_t(a) \xi, \eta \ra$ is differentiable at $t =0$ there exists a positive $\d$ such that $$ |t| < \d \Rightarrow |\la\big((\a_t(a) - a)/t\big)\xi, \eta\ra - g^{\prime} (0) | < 1.$$ 
Hence for $0 < |t| < \d$  we have   $|\la\big((\a_t(a) - a)/t\big)\xi, \eta\ra| < 1 +  |g^{\prime} (0)|.$ 

\noindent
For $|t| \geq \d$ we have  $|\la\big((\a_t(a) - a)/t\big)\xi, \eta\ra| < (2/\d)\|a\|\|\xi\|\|\eta\|$  so 

\noindent
$\underset{t \neq 0}{\sup}   |\la\big((\a_t(a) - a)/t\big)\xi, \eta\ra| < \infty,$ then by the uniform boundedness principle  $\underset{t \neq 0}{\sup}   \|(\a_t(a) - a)/t\| < \infty,$ and $a$ is Lipschitz continuous. 

Proof of (iii) $\Rightarrow $ (iv):

If $a$ is Lipschitz continuous then there exists a positive real $k$ and a sequence $(t_n)$ of non-zero reals such that $t_n \to 0$ and $\|(\a_{t_n}(a) -a)/t_n\|\leq k.$ Then the arguments at the top of Subsection \ref{BdForm} imply that 
 \begin{align*} \forall \xi, \eta \in \text{dom}(D): \quad |S(i[D,a])(\xi, \eta) | = & \underset{t \to 0}{\lim}|\la(\a_t(a)-a)\xi, \eta\ra/t| \\
 = & \underset{n \to \infty}{\lim}|\la(\a_{t_n}(a)-a)\xi, \eta\ra/t_n| \\ & \leq k \|\xi\|\|\eta\| \text{ so } \|S(i[D,a])\| \leq k. \end{align*}

Proof of (iv) $\Rightarrow $ (v):

Given (iv) then there exists a bounded operator $b$ on $H$ such that for $c,r \in \bz,$ $\xi \in e_cH$ and $\eta \in e_rH$ we have 
\begin{align*} \la b \xi , \eta \ra = & S(i[D,a])(\xi, \eta ) \\ 
= & i\la a\xi, D\eta \ra - i\la aD\xi, \eta\ra \\ =&  
\la i( d_re_rae_c - e_rae_cd_c ) \xi, \eta\ra. \end{align*}

Hence the matrix $m(i[D,a]) $ is the matrix of the bounded operator $b.$ 

Proof of (v) $\Rightarrow $ (vi):

As in the concrete example above we define a core $\ce$ for $D$ by $$\ce:= \mathrm{span}\big( \cup_{n \in \bz} e_nH\big).$$
Suppose there is a bounded operator $b$ such that for any pair $r, c$ in $\bz$ we have $ e_rbe_c = i(d_re_rae_c - e_rae_cd_c),$ then for $\xi$ in $e_cH$ and $ \eta $ in $e_rH$ we get $$\la b \xi, \eta \ra = i\la  (d_re_rae_c - e_rae_cd_c)\xi, \eta \ra = i\la a\xi, D\eta \ra -i \la aD\xi, \eta \ra.$$ By linearity and rearrangement  we get 
$$ \forall \xi, \eta \in \ce : \quad |\la a \xi, D \eta \ra| \leq (\|b\|\|\xi\|+ \|aD\xi\|) \|\eta\| .$$ For a fixed $\xi $ in $\ce$ this implies that $a\xi $ is in the domain of definition for $(D\big|\ce)^*$ the adjoint of the  restriction of $D$  to $\ce$. Now $\ce$ is a core for $D$ so  $(D\big|\ce)^* = D$  and we get that $a\ce \subseteq $ dom$(D),$ which then shows that $i(Da-aD)$ is defined on the core $\ce $ and equals the bounded operator $b$ on that space.   

Proof of (vi) $\Rightarrow $ (vii):

Suppose that $Da-aD$ is defined on a core $\cf$ for $D$ and that there exists a bounded operator $b$ such that $i(Da-aD) = b\big|\cf.$ Then since $\cf$ is a core for $D$ we can for any vector $\xi$ in dom$(D)$ find a sequence of vectors $(\phi_n ) $ in $\cf$ such that $$ \phi_n \to  \xi \text{ and } D\phi_n \to D\xi \text{ for } n \to \infty.$$
Since   $b\phi_n$ converges towards $b\xi,$ $a\phi_n$ converges towards $a\xi,$ $aD\phi_n$ converges towards $aD\xi,$ and $a\phi_n$ is in dom$(D)$  we get by linearity 
$$ a\phi_n \to a\xi \text{ and } Da\phi_n \to - i b \xi + aD\xi.$$

Hence $a\xi $ is in dom$(D)$ and $b \xi = i(Da-aD)\xi.$

Proof of (vii) $\Rightarrow $ (i): 
Let us first recall that the unitary representation $u_t$ of $\br$ is strongly continuous.
Let   $a$ be a bounded operator which satisfies  condition (vii), let $b$ denote the bounded operator such that the restriction $b\big|$dom$(D)$ equals $i(Da-aD)$  and let $\xi_0$ denote a vector in dom$D,$  then we will show that the function $\a_t(a)\xi_0$ is differentiable at $t=0$ with derivative $b\xi_0.$ Since $\xi_0 $ is in dom$(D)$ so is $a\xi_0$ and we have  
$$\underset{t \to 0 }{\lim}(1/t)(u_{-t} - I)\xi_0 = -iD\xi_0 \text{ and } \underset{t \to 0 }{\lim}(1/t)(u_{t} - I)a\xi_0 = iDa\xi_0.$$ We have $b\xi_0 = i(Da-aD)\xi_0, $ so we may estimate  
\begin{align*}
\forall t \neq 0: \,  \|(1/t) ( \a_t(a) - a )\xi_0 - b\xi_0\| & \leq \|u_ta\big((1/t)(u_{-t} -I)\xi_0 + iD\xi_0 \big) \| \\ &+ \|(1/t)(u_t-I)a\xi_0 -iDa\xi_0\| \\ &+ \|i(I-u_t)aD\xi_0\| \end{align*} 

By the equalities above and the strong continuity of $u_t,$ the inequality show that $\underset{t\to 0}{\lim}\|(1/t) ( \a_t(a) - a )\xi_0 - b\xi_0\| = 0,$ and we have established the strong differentiability for a vector $\xi_0 $ in dom$(D).$ 

Let now $s$ be any real then since $u_s = e^{isD}, $ the unitary $u_s$ commutes with $D$ and  then $u_s\mathrm{dom}(D) = \mathrm{dom}(D).$ Hence $\a_s(a)\mathrm{dom}(D) \subseteq \mathrm{dom}(D), $ and the domain of definition is  dom$(D)$ in the equations below
  $$i(D\a_s(a) - \a_s(a)D) =u_s(i(Da - aD))u_{-s}= \a_s(b)\big|\mathrm{dom}(D).$$ We find that $\a_s(a)$ also satisfies condition (vii) and that the closure of the commutator $i[D,\a_s(a)]$ is $\a_s(b). $ We will then recall the beginning of the proof of (vii) implies (i) and note that for $\xi_0$ in dom$(D)$ the function $\a_s(a)\xi_0$ is differentiable everywhere with derivative $\a_s(b)\xi_0.$ 

We will now show that we can replace $\xi_0$ by any vector in $H$ and then establish the strong $D-$differentiability of $a.$ We will do so by constructing $\a_s(a)\xi$ as an integral, but in order to be able to integrate we need some continuity properties. 
 Since the function $t \to u_t = e^{itD}$ is strongly continuous and $u_t$ is unitary, we can for a vector $\xi$ in $H$ and real numbers $s, t$ estimate \begin{align*} \|\a_{t+s}(b)\xi  - \a_{s}(b)\xi\| &\leq \|\a_{t+s}(b)(I - u_t)\xi)\| + \| (u_t -I)\a_s(b) \xi\|\\  &\leq \|b\|\|(u_t -I)\xi\| + \| (u_t -I)\a_s(b) \xi\|. \end{align*}
Hence the bounded function $s \to \a_s(b)\xi $  from $\br $ to $H$  is continuous and then Riemann integrable on any bounded interval, so we can define an operator valued  function $c(t) $  by 
$$c(t)\xi := a\xi + \int_0^t \a_s(b)\xi ds.$$ 

 By construction $c(t)$ is a bounded operator, the function $t \to  c(t)\xi $ from $\br$ to $H$ is norm differentiable and satisfies $$c(0)\xi = a\xi \text{ and   } \frac{d}{ds}(c(s)\xi) = \a_s(b)\xi.$$
For a vector $\xi_0$ in dom$(D)$ we know that the function $\a_s(a)\xi_0$ is norm differentiable and satisfies  
$$\a_0(a)\xi_0 = a\xi_0  \text{ and   } \frac{d}{ds}(\a_s(a)\xi_0 )= \a_s(b)\xi_0.$$
Hence we conclude that for $\xi_0$ in dom$(D)$ and $s$ in $\br$ we have $c(s)\xi_0 = \a_s(a)\xi_0.$ Since both of the operators $c(s)$ and $\a_s(a)$ are bounded we get that $c(s) = \a_s(a),$ and then
$$ \forall \xi \in H: \underset{t \to 0}{\lim}(1/t)(\a_t(a) -a)\xi  = b\xi,$$
so $a $ is strongly $D-$differentiable.  

Suppose that $a$ is a weakly $D-$differentiable operator, then for any pair of vectors $\xi, \eta$ we have 
$\la \d_w^D(a)\xi,\eta \ra = \underset{t \to 0}{ \lim}\la (\a_t(a) -a)\xi, \eta \ra/t .$   The space dom$(D)$ is invariant under $a$  and $\d_w^D(a)\big|$dom$(D) = i(Da-aD).$ The proof of (vii) $\Rightarrow$ (i) shows that $$|\la( \a_t(a)-a) \xi,\eta\ra| = |\int_0^t\la \a_s(\d_w^D(a))\xi, \eta \ra ds| \leq \|\d_w^D(a)\|\|\xi\|\|\eta\||t|,$$ so $\|\a_t(a) - a\| \leq \|\d_w^D(a)\||t|$ and the theorem follows.  
\end{proof}

In analogy with the properties for the uniform $D-$derivation  we have the following theorem. 

\begin{thm} \label{WAlg}
The domain of definition $\mathrm{dom}(\d_w^D)$ is a strongly dense *-subalgebra of $B(H)$ and $\d_w^D$ is a *-derivation into $B(H)$  which extends $\d_u^D.$  
The graph of $\d_w^D$ is  weakly  closed and $\mathrm{dom}(\d_w^D)$   is a Banach *-algebra under the norm 
$\|a\|_D := \|a\| + \|\d_w^D(a)\|.$ 
\end{thm}

\begin{proof}
Suppose that $a, b $ are both weakly $D-$differentiable operators then by Theorem \ref{WD} \begin{align*}\|\a_t(ab) - ab\| \leq & \|\a_t(a)\|\|\a_t(b) -b\| + \|\a_t(a) -a\|\|b\|\\ \leq & \big(\|a\|\|\d_w^D(b)\| +\|\d_w^D(a)\|\|b\|\big)|t|,\end{align*}
so $ab$ is $D-$Lipschitz continuous and hence  the domain of $\d_w^D$ is an algebra, and clearly a *-algebra too.  Since any weakly $D-$differentiable operator leaves dom$(D)$ invariant we have for any $\xi $ in dom$(D)$ that 
\begin{align*}
\d_w^D(ab)\xi = & i(Dab -abD)\xi \\=& i(Da-aD)b\xi + a(i(Db-bD))\xi \\ =& (\d_w^D(a)b + a\d_w^D(b))\xi,\end{align*} and then $\d_w^D $ is a derivation. Since uniformly $D-$differentiable operators are weakly differentiable we get that $\d_w^D$ extends $\d_u^D,$ so its domain of definition is dense in the strong operator topology on $B(H)$ too. 

We will now show that the graph of $\d_w^D$ is  closed in the weak operator topology on $B(H)\oplus B(H).$ It is quite easy to prove, but we find the result surprising, since it is our general impression,  that differentiability does not play well together with the weak operator topology. We first remark that the graph of $\d_w^D$ is a linear space so by \cite{GKP} page 173, the strong operator closure of the graph equals the weak operator closure, and it is sufficient to prove that the graph is a strong operator closed subspace of $B(H) \oplus B(H).$  
Let $\big( (a_{\iota}, \d_w^D(a_{\iota}))\big)_{(\iota \in J)}$ be a strongly convergent net in the graph of $\d_w^D$ with limit point $(x,y).$ Then for any $\xi$ in dom$(D)$ the net  $\big(a_{\iota}\xi\big)_{(\iota \in J)} $ converges towards $x\xi,$ the net $\big(\d_w^D(a_{\iota})\xi\big)_{(\iota \in J)}$ converges towards $y\xi$ and the net  $\big(a_{\iota}D\xi\big)_{(\iota \in J)} $ converges towards $x D \xi.$ Since each $a_{\iota} $ is weakly $D-$differentiable we know that $a_{\iota} \xi $ is in dom$(D)$ and 
$\d_w^D(a_{\iota})\xi = i(Da_{\iota}-a_{\iota}D)\xi.$ By rearrangement we get $$Da_{\iota}\xi = -i\d_w^D(a_{\iota} )\xi + a_{\iota}D\xi \to -iy\xi
+ xD\xi \text{ through } \iota \in J.$$
Since $D$ is a closed operator we find that $$x\xi \in \mathrm{dom}(D) \text{ and } y\xi = i(Dx - xD)\xi.$$
Then by Theorem \ref{WD} we have that $x$ is weakly differentiable and $y = \d_w^D(x),$ so the graph of $\d_w^D$ is strongly closed and then also weak operator closed. 

For weakly $D-$differentiable operators $a,b$ we have $$ \|ab\|_D = \|ab\| +\|\d_w^D(a)b + a\d_w^D(b)\| \leq \|a\|_D \|b\|_D $$
 and the weak operator  closedness of the graph of $\d_w^D$  implies that  $\big(\mathrm{dom}(\d_w^D), \|.\|_D \big)$ is a Banach *-algebra. The theorem follows.  
\end{proof}

The following result is well known, but we present it here, because it shows how it is possible to extend the Banach algebra property, we just established,  to algebras which are obtained  as intersections of domains of unbounded derivations or domains of powers of such derivations. A situation like this is quite often included in the setup in noncommutative geometry.

Let $A$ be a norm closed subalgebra of $B(H)$ and  $\d_1, \dots, \d_k$ closed derivations  defined on subalgebras  of $A$ and with values in $B(H).$  Let $n_1, \dots, n_k$ be natural numbers and for $1  \leq j \leq k $ let  $$\ca_j:=  \mathrm{dom}(\d_j^{n_j})\, \text{ and }\, \forall a \in \ca_j:  \|a\|_{(j,n_j)} := \sum_{i=0}^{n_j}\frac{1}{i!}\|\d_j^i(a)\|.$$ Finally we define 
$$\ca := \cap_{j=1}^k \ca_j \text{ and } \forall a \in \ca: \, \||a|\|:=  \underset{1 \leq j \leq k}{\max}\|a\|_{(j,n_j)}.$$

\begin{pro}  
The normed space $(\ca, \||.|\|)$ is a Banach algebra. 
\end{pro} 
\begin{proof} Since the domain of definition for any $\d_j$ is an algebra we get by induction  that for any pair $a,b$ in dom$(\d_j^i)$ that the product $ab $ also belongs to that domain and 
$$\d_j^i(ab) = \sum_{k=0}^i \binom{i}{k}\d_j^k(a)\d_j^{i-k}(b).$$ 
That property then gives 
$$\|ab\|_{(j,n_j)} \leq \sum_{l=0}^{n_j} \sum_{k=0}^{n_j} \frac{1}{l!}\frac{1}{ k!}\|\d_j^l(a)\|\|\d_j^k(b)\| = \|a\|_{(j,n_j)}\|b\|_{(j,n_j)}.$$

It then follows that the norm $\||a|\| $ on $\ca$ is submultiplicative and the closedness assumptions on the derivations $\d_j$ will yield the completeness property of the normed algebra, so we have a Banach algebra.
\end{proof}

\section{uniform versus weak $D-$differentiability} 

We know that any bounded uniformly $D-$differentiable operator on $H$ is weakly $D-$differentiable, and the example in subsection  \ref{WneqU}
de-monstrates that there exists weakly $D-$differentiable operators, which are not uniformly $D-$differentiable. Above we showed that the graph of $\d_w^D$ is weak operator  closed, and this indicates in some imprecise way that there must be many more weakly $D-$differentiable operators than uniformly $D-$differentiable ones. Here we will present a condition which characterizes those weakly $D-$differentiable  operators that are uniformly $D-$differentiable.

\begin{thm} \label{CSD}
Let $a$ be a bounded operator on $H$ then $a$ is uniformly
$D-$differentiable if and only if it is weakly $D-$differentiable and
the function $t \to \alpha_t(\d_w^D(a))$ is norm continuous.
\end{thm}

\begin{proof} 
Let us first assume that $a$ is uniformly $D-$differentiable and we find as in the proof of item (ii) of  Theorem \ref{du} that there exists a function  $\eps(t)$ with values in $B(H)$
which satisfies $\eps(0) = 0,$ is continuous at $0$ and
$$\alpha_t(a) = a + t\d_u^D(a) + t\eps(t).$$ We may repeat some of  the arguments from that proof and we obtain from Equation \ref{bCont} that 
$$\alpha_t(\d_u^D(a)) - \d_u^D(a) = \eps(t) -
\alpha_t(\eps(-t)),$$ so the function $\alpha_t(\d_u^D(a))$ is norm continuous at $t = 0.$  As in the previous proof we find that this continuity at $t=0$ may be extended to any real $t.$

Now assume that $a$ is a weakly $D-$differentiable operator satisfying the condition
that $t \to \alpha_t(\d_w^D(a)) $ is norm continuous.
The proof of Theorem \ref{WD}, (vii) $\Rightarrow$ (i), can then be copied, and we define   a function $c(t) : \br \to B(H)$ by $$ c(t) \, := \, a +  \int_0^t
\alpha_s(\d_w^D(a)) ds.$$ Then $c(t) $ is norm differentiable
everywhere with $c^{\prime}(t) =  \alpha_t(\d_w^D(a)), $  and for  any
pair of vectors $\xi, \eta $  the scalar valued function
$g(t):= \la c(t) \xi , \eta \ra$ is then differentiable everywhere and
we have $g^{\prime}(t) =  \la \alpha_t(\d_w^D(a))\xi, \eta \ra.$ On
the other hand the weak differentiability of $a$ implies that the scalar valued 
function $h(t) \, := \, \la \alpha_t(a) \xi , \eta \ra $ is 
differentiable with derivative $h^{\prime}(t) =  \la
\alpha_t(\d_w^D(a)) \xi, \eta \ra.$ Hence $g^{\prime}(t) =
h^{\prime}(t)$ and since we also have $g(0) = h(0) = \la a
\xi, \eta \ra $ we have $g(t) = h(t)$ and then $c(t) = \alpha_t(a),$
so $a$ is uniformly $D-$differentiable.
\end{proof}

\begin{cor}
Let $n$ be a natural number and $a$  an  $(n+1)-$times weakly $D-$differentiable operator in $B(H).$ Then it is
$n-$times uniformly $D-$differentiable.
\end{cor}

\begin{proof}
The operator $(\d_w^D)^n(a)$ is weakly differentiable and then by Theorem \ref{WD}, $(\d_w^D)^n(a)$   is $D-$Lipschitz continuous and in particular the function $t \to \a_t((\d_w^D)^n(a)) $ is norm continuous. The operator $(\d_w^D)^{n-1}(a)$ is weakly $D-$differentiable such that $\a_t(\d_w^D(\d_w^D)^{n-1}(a)))$ is norm continuous, then $(\d_w^D)^{n-1}(a)$ is uniformly $D-$differentiable. Then we repeat the argument and we find  that all the uniform derivatives $(\d_u^D)^j(a)$ exist for $1 \leq j \leq n.$  
\end{proof} 

\begin{exa} \label{No} \end{exa}
The contents of Theorem \ref{CSD} and Example \ref{WneqU} show that the  function $t \to \a_t(M_{\mathrm{sign}(\theta)}) $ is not norm continuous, and hence $\a_t$ is not strongly continuous. We will like to make this observation explicit by pointing out that \begin{displaymath}
\|\a_t(M_{\mathrm{sign}(\theta)})- M_{\mathrm{sign}(\theta)}\|  = \| \mathrm{sign}(\theta)_t - \mathrm{sign}(\theta)\|_{\infty} = \begin{cases} 0 \text{ if } t = 2n\pi \\ 2 \text{ if } t \neq  2n\pi
\end{cases}
\end{displaymath}
As a  consequence of this and the proof of Theorem \ref{WD}, (vii) $\Rightarrow$ (i)  we get that the function $t \to \a_t(M_{\mathrm{sign}(\theta)})$ in $B(H) $ is not Riemann  integrable. If it were, then we would be able to show that  $$\a_t(M_{|\theta|})  = M_{|\theta|} +  \int_0^t\a_s(M_{\mathrm{sign}(\theta)})ds$$ which in turn shows that $M_{|\theta|} $ would be uniformly $D-$differentiable.

Another consequence is that for $D = \frac{1}{i}\frac{d}{d\theta} $ the domain  $\cd_u^D $ is not norm dense in $B(H).$ The following elementary application of the triangle inequality shows that no bounded operator $a $ such that $\|a - M_{\mathrm{sign}(\theta)}\| \leq \d < 1$ is uniformly $D-$differentiable. For such an operator   and a real $t$ with $0 < |t|< \pi$ we have   $$ \|a - \a_t(a)\| \geq 
\| \mathrm{sign}(\theta)_t -  \mathrm{sign}(\theta)\|_{\infty} - 2 \|M_{\mathrm{sign}(\theta)} -a\| \geq  2(1 - \d).$$

This raises the question if the algebra $\cd^D_w$ is norm dense in $B(H) ?$

\end{document}